\documentclass[twoside,reqno]{amsart} 
\pdfoutput=1
\usepackage[top=28mm,right=28mm,bottom=28mm,left=28mm]{geometry}
\usepackage[T1]{fontenc}
\usepackage{microtype}
\usepackage{amsfonts, amsmath, amssymb,stmaryrd,array, enumerate}
\usepackage{hyperref}

\newcommand{\Aut}{\operatorname{Aut}}
\newcommand{\Inn}{\operatorname{Inn}}
\newcommand{\Sym}{\operatorname{Sym}}
\renewcommand{\leq}{\leqslant}
\renewcommand{\geq}{\geqslant}
\newcommand{\leqn}{\trianglelefteqslant}

\renewcommand{\:}{\colon}

\theoremstyle{plain}
\newtheorem{theorem}{Theorem}[section]
\newtheorem{lemma}[theorem]{Lemma}
\newtheorem{proposition}[theorem]{Proposition}
\newtheorem{thm}{Theorem}
\newtheorem{con}[thm]{Conjecture}

\begin{document}

	\title[Kronecker classes, normal coverings and chief factors of groups]{Kronecker classes, normal coverings \\ and chief factors of groups}
	
	\author{Marco Fusari}
	\address{M. Fusari, Dipartimento di Matematica ``Felice Casorati'', University of Pavia, 27100 Pavia, Italy}
	\email{lucamarcofusari@gmail.com}
	
	\author{Scott Harper}
	\address{S. Harper, School of Mathematics and Statistics, University of St Andrews, St Andrews, KY16 9SS, UK}
	\email{scott.harper@st-andrews.ac.uk}
	
	\author{Pablo Spiga}
	\address{P. Spiga, Dipartimento di Matematica e Applicazioni, University of Milano-Bicocca, 20126 Milano, Italy}
	\email{pablo.spiga@unimib.it}
	
	\begin{abstract}
		For a group $G$, a subgroup $U \leq G$ and a group $\mathrm{Inn}(G) \leq A \leq \mathrm{Aut}(G)$, we say that $U$ is an $A$-covering group of $G$ if $G = \bigcup_{a\in A}U^a$. A theorem of Jordan (1872) implies that if $G$ is a finite group, $A = \mathrm{Inn}(G)$ and $U$ is an $A$-covering group of $G$, then $U = G$. Motivated by a question concerning Kronecker classes of field extensions, Neumann and Praeger (1988) conjectured that, more generally, there is an integer function $f$ such that if $G$ is a finite group and $U$ is an $A$-covering subgroup of $G$, then $|G:U| \leq f(|A:\mathrm{Inn}(G)|)$. A key piece of evidence for this conjecture is a theorem of Praeger (1994), which asserts that there is a two-variable integer function $g$ such that if $G$ is a finite group and $U$ is an $A$-covering subgroup of $G$, then $|G:U|\leq g(|A:\mathrm{Inn}(G)|,c)$ where $c$ is the number of $A$-chief factors of~$G$. Unfortunately, the proof of this result contains an error. In this paper, using a different argument, we give a correct proof of this theorem.
	\end{abstract}
	
	\dedicatory{Dedicated to Cheryl Praeger: \\ for her invaluable contributions, which continue to inspire, entertain and challenge us.}  
	
	
	
	\maketitle

	\section{Introduction}\label{s:intro}

	There are some remarkable connections between normal coverings of groups and algebraic number fields. In particular, various open conjectures regarding Kronecker classes of field extensions can be reformulated as questions about covering groups by conjugates of subgroups. The interplay between these two subjects is detailed in \cite{Klingen98} and summarised in \cite[Section~1.7]{BSW24}; we also refer the reader to the papers \cite{Jehne77,Klingen78,Praeger88,Praeger94}. For applications to Kronecker classes of field extensions, the most important open problem is the following conjecture of Neumann and Praeger \cite[Conjecture~4.3]{Praeger94} (see also the Kourovka Notebook \cite[11.71]{Kourovka}).
	
	\begin{con}[Neumann \& Praeger, 1990] \label{con:neumann-praeger}
		There exists a function $f\: \mathbb{N} \to \mathbb{N}$ such that the following holds. Let $G$ be a finite group, let $U \leq G$ and let $\Inn(G) \leq A \leq \Aut(G)$ with $|A:\Inn(G)| = n$. If $G = \bigcup_{a \in A} U^a$, then $|G:U| \leq f(n)$.
	\end{con}
	
	This conjecture remains open. Establishing the veracity of this conjecture has important applications not only in number theory, but also in investigating Erd\H{o}s--Ko--Rado-type of theorems in algebraic combinatorics (see for instance~\cite[Section~1.4 and~1.6]{BSW24} or~\cite[Section~1.1]{FPS24}). This conjecture can also be seen as a vast generalisation of the classical theorem of Jordan \cite{Jordan72} that a finite group cannot be covered by conjugates of a proper subgroup.
	
	In 1994, Praeger gave two significant pieces of evidence for Conjecture~\ref{con:neumann-praeger}. The first is a proof of Conjecture~\ref{con:neumann-praeger} when $U$ is a maximal subgroup of $G$ \cite[Theorem~4.4]{Praeger94}, and the second is Theorem~\ref{thm:main} below, which, roughly speaking, establishes Conjecture~\ref{con:neumann-praeger} when $G$ has bounded $A$-chief length \cite[Theorem~4.6]{Praeger94}. To state this more precisely, we need some notation.
	
	Let $G$ be a finite group and let $A$ be a subgroup of the automorphism group $\Aut(G)$. Then the \emph{$A$-core} of $U$, written $U_A$, and the \emph{$A$-cocore} of $U$, written $U^A$, are defined as follows
	\[
	U_A = \bigcap_{a \in A} U^a \qquad \text{and} \qquad U^A = \bigcup_{a \in A} U^a.
	\]
	An \emph{$A$-chief series} for $G$ is an unrefinable $A$-invariant series $G = G_0 > \cdots > G_c = 1$. By the Jordan--H\"older theorem, every $A$-chief series for $G$ has the same length, called the \emph{$A$-chief length} of $G$, written $\ell_A(G)$.
	
	We are now in a position to state \cite[Theorem~4.6]{Praeger94}.
	
	\begin{thm} \label{thm:main}
		There exists a function $g\: \mathbb{N} \times \mathbb{N} \to \mathbb{N}$ such that the following holds. Let $G$ be a finite group, let $U \leq G$ and let $\Inn(G) \leq A \leq \Aut(G)$ with $|A:\Inn(G)| = n$. Write $c = \ell_A(G/U_A)$. If $G = \bigcup_{a \in A} U^a$, then $|G:U| \leq g(n,c)$.
	\end{thm}  

	Unfortunately, the proof of~\cite[Theorem~4.6]{Praeger94} is not correct. Indeed, there is an error on page~29 (on line~7 counting from the bottom of the page). The mistake is subtle. Using the notation from the proof of~\cite[Theorem~4.6]{Praeger94}, it occurs in the inductive step under the assumption that $\overline{G} = Z_p^k$. In the degenerate case when $\overline{G}=\overline{U}$, the inequality $p \leq |\overline{G}:\overline{U}|$ is false and no upper bound on $p$ is derived in this case. To the best of our knowledge, this error was first observed by the third author in the preparation of~\cite{FPS24}. 
		
	The purpose of this paper is to give a complete proof of Theorem~\ref{thm:main}. We were unable to fix the precise error in Praeger's argument, so we give an alternative proof.
	
	As a consequence of how we structure our proof, we actually also verify another case of Conjecture~\ref{con:neumann-praeger}, namely when $U$ supplements a minimal $A$-invariant subgroup of $G$, see Proposition~\ref{prop:supplement} below.

	\section{Preliminaries} \label{s:prelims}
	
	\subsection{Diagonal subgroups} \label{ss:p_diagonal}
	
	Let $T$ be a finite group and let $k$ be a positive integer. It will be useful to introduce notation for some particular subgroups of $T^k$. For $j \in \{1,\dots,k\}$, we write 
	\[
	T_j = \{ (x_1, \dots, x_k) \in T^k \mid \text{$x_i = 1$ if $i \neq j$} \},
	\]
	and more generally, for $\Delta \subseteq \{1,\dots,k\}$, we write 
	\begin{equation} \label{eq:diagonal}
		T_\Delta = \{ (x_1, \dots x_k) \in T^k \mid \text{$x_i = 1$ if $i \not\in \Delta$} \},
	\end{equation}
	so $T_\emptyset=1$, $T_{\{i\}} = T_i$ and $T_{\{1,\dots,k\}} = T^k$. A subgroup $H \leq T^k$ is called a \emph{diagonal subgroup} if there exist $\phi_1, \dots, \phi_k \in \Aut(T)$ such that $H = \{ (x^{\phi_1}, \dots, x^{\phi_k}) \mid x \in T \}$.
	
	\begin{lemma} \label{lem:char_simple}
		Let $T$ be a finite nonabelian simple group, let $k$ be a positive integer, let $A \leq \Aut(T^k)$ and let $U < T^k$ such that $U^A = T^k$. Then there is a partition $\{ \Delta_1, \dots, \Delta_t \}$ of $\{1, \dots, k\}$ such that $U = D_1 \times \cdots \times D_t$, where $D_i$ is a diagonal subgroup of $T_{\Delta_i}$ for all $1 \leq i \leq t$.
	\end{lemma}
	
	\begin{proof}
		Fix $1 \leq i \leq k$ and let $\pi_i\: G \to T$ be the projection $(x_1, \dots, x_k) \mapsto x_i$ onto the $i$th direct factor of $T^k$. We claim that $U\pi_i = T$. To see this, let $x \in T$ be arbitrary and consider $g = (x, \dots, x) \in T^k$. Since $U^A = T^k$, there exists $a \in A$ such that $g^a \in U$. Since $\Aut(T^k) = \Aut(T) \wr \Sym_k$, there exists $b \in \Aut(T)$ such that the $i$th component of $g^a$ is $x^b$. In particular, $x$ is $\Aut(T)$-conjugate to an element of $U\pi_i$. Therefore, 
		\[
		T = \bigcup_{a\in \Aut(T)}(U\pi_i)^a= (U\pi_i)^{\Aut(T)}.
		\] 
		However, by \cite[Proposition~2]{Saxl88}, this forces $U\pi_i = T$, as claimed. The result now follows by Goursat's lemma.
	\end{proof}

	\subsection{A bound on the order of a finite simple group} \label{ss:p_bound}
	
	For the rest of the paper, fix a nondecreasing function $h\: \mathbb{N} \to \mathbb{N}$ such that for all finite nonabelian simple groups $T$, if the number of $\Aut(T)$-classes of elements of $T$ is $m$, then $|T| \leq h(m)$. By \cite[Lemma~4.4]{Pyber92}, such a function exists, and, as explained in the comments following \cite[Theorem~4.5]{Praeger94}, one can take $h(m) = 2^{c \, (\log{m})^2 \log\log m}$ for an appropriate constant $c$.

	\section{Proof of the main result} \label{s:proof}
	
	Before proving Theorem~\ref{thm:main}, we begin with a special case that arises in the proof, namely, in the notation of Theorem~\ref{thm:main}, when the subgroup $U$ supplements a minimal $A$-invariant subgroup of $G$. In this case, we can actually prove the more general Conjecture~\ref{con:neumann-praeger}, so this may be of independent interest.
	
	\begin{proposition}\label{prop:supplement}
		There exists a function $f\: \mathbb{N}\to \mathbb{N}$ such that the following holds. Let $G$ be a finite group, let $U \leq G$ and let $\Inn(G) \leq A \leq \Aut(G)$ with $|A:\Inn(G)| = n$. Assume that there exists a minimal $A$-invariant subgroup $L$ of $G$ such that $G = UL$. If $G = \bigcup_{a \in A} U^a$, then $|G:U| \leq f(n)$.
	\end{proposition}

	\begin{proof}
		Let $h\: \mathbb{N} \to \mathbb{N}$ be the function from Section~\ref{ss:p_bound}, and let $f:\mathbb{N}\to\mathbb{N}$ be defined as 
		\[
		f(n)=\max\{n, \, h(n)^{n^2}\}.
		\]
		
		Assume that $G = \bigcup_{a \in A} U^a$. If $L \leq U$, then $U=UL=G$ and $|G:U|=1 \leq f(n)$, so the result holds. Therefore, for the rest of the proof we may assume that $L\not\leq U$. In particular, $U\cap L < L$, so the $A$-core $(U\cap L)_A$ is trivial since $L$ is a minimal $A$-invariant subgroup of $G$.
		
		First assume $L$ is abelian. Certainly $U \cap L \leqn U$ and, since $L$ is abelian, we also have $U \cap L \leqn L$, so $U \cap L \leqn UL = G$. Therefore, the set $\{ (U \cap L)^a \mid a \in A \}$ has size at most $|A:\Inn(G)| = n$. Since $L = (U \cap L)^A$, we deduce that $|L| \leq n |U \cap L|$, so 
		\[
		|G:U| = |UL:U| = |L:U \cap L| \leq n \leq f(n).
		\]
		
		Next assume $L$ is nonabelian. Then $L = T^k$ for a positive integer $k$ and a nonabelian simple group $T$. Fix a partition $\mathcal{S} = \{\Sigma_1, \dots, \Sigma_r\}$ of $\{1, \dots, k\}$ such that 
		\[
		\{T_i \mid i \in \Sigma_1\}, \dots, \{ T_i \mid i \in \Sigma_r\}
		\]
		are the $G$-orbits on the set 
		\[
		\Omega=\{T_1, \dots, T_k\}
		\] 
		of simple factors of $L$. Since $G = UL$ and $L$ acts trivially on $\Omega$, the $G$-orbits are also $U$-orbits. Since $L$ has no proper nontrivial $A$-invariant subgroups, $A$ acts transitively on the set $\Omega$, and since $\Inn(G) \leqn A$, the $G$-orbits form a system of imprimitivity for $A$. This means that $r$ divides $k$, and writing $k=rs$, we have $|\Sigma_i| = s$ for all $1 \leq i \leq r$. Moreover, $r$ divides $n = |A:\Inn(G)|$, so, in particular, 
		\begin{equation} \label{eq:bound_rn}
			r \leq n.
		\end{equation} 
		For ease of notation, we will assume that $\Sigma_i = \{(i-1)s+1, \dots, is \}$ for all $1 \leq i \leq r$. Moreover, we will write elements of $L = T^k$ in the form $(x_{11}, \dots, x_{1s} \mid \ldots \mid x_{r1}, \dots, x_{rs})$, where $x_{ij} \in T_{(i-1)s+j}$. We will refer to $T_{\Sigma_1}$, \dots, $T_{\Sigma_r}$ as the \emph{blocks} of $L$. (Recall from~\eqref{eq:diagonal} that $T_{\Sigma_i}=\prod_{j \in \Sigma_i}T_j$.)
		
		Since $(U \cap L)^A = L$, by Lemma~\ref{lem:char_simple}, we may fix a partition $\mathcal{D} = \{ \Delta_1, \dots, \Delta_t \}$ of $\{1, \dots, k\}$ such that 
		\[
		U \cap L = D_1 \times \cdots \times D_t,
		\] 
		where $D_i$ is a diagonal subgroup of $T_{\Delta_i}$ for all $1 \leq i \leq t$. Since $U\cap L<L$, for at least one $1 \leq i \leq t$ we have $|\Delta_i| \geq 2$. By relabelling the parts of $\mathcal{D}$, without loss of generality, we may assume that $|\Delta_1| \geq 2$.
		
		Let $m$ be the number of $\Aut(T)$-classes in $T$. Suppose that $m > r$. Then we may fix nontrivial elements $x_1, \dots, x_r \in T$ that represent distinct $\Aut(T)$-classes. Consider 
		\[
		g = (x_1,1,\dots,1 \mid x_2,1,\dots,1 \mid \ldots \mid x_r, 1, \dots, 1) \in T_{\Sigma_1}\times T_{\Sigma_2}\times\cdots\times T_{\Sigma_r}= L.
		\] 
		Since $(U \cap L)^A = L$, there exists $a \in A$ such that $g^a \in U \cap L$. Fix $l \in \Delta_1$. Since $U$ acts transitively on the set of factors of each block of $L$ (that is, $U$ acts transitively by conjugation on the simple direct factors of $T_{\Sigma_j}$ for each $j$), there exists $u \in U$ such that the $l$th entry of $g^{au}$ is nontrivial. Now $g^{au} \in U \cap L$, since $U \cap L \leqn U$, so $D_1$ contains an element with a nontrivial $l$th entry. However, all nontrivial entries of $g$, and hence $g^{au}$, are pairwise not $\Aut(T)$-conjugate, so $|\Delta_1| = 1$, which is a contradiction. This contradiction implies
		\begin{equation} \label{eq:bound_mr}
			m \leq r.
		\end{equation}
		
		Suppose that $s > r$.  Let $x \in T$ be nontrivial and consider
		\begin{align*}
			g &= (x,1,\dots,1 \mid \ldots \mid x, 1, \dots, 1) \in L, \\
			h &= (1,\dots,1 \mid x,1,\dots,1 \mid x,x,1,\dots,1 \mid \ldots \mid x, \dots x, 1, \dots, 1) \in L.
		\end{align*}
		Since $(U \cap L)^A = L$, there exist $a,b \in A$ such that $g^a, h^b \in U \cap L$. To prove our claim, we will need to establish two secondary claims.
		
		First, let $1 \leq i \leq t$. We claim that $|\Delta_i \cap \Sigma_j| \leq 1$ for all $1 \leq j \leq r$. Fix $l \in \Delta_i$. As before, there exists $u \in U$ such that the $l$th entry of $g^{au}$ is nontrivial and $g^{au} \in U \cap L$, so $D_i$ contains an element with a nontrivial $l$th entry. However, in each block, exactly one entry of $g$, and hence $g^{au}$, is nontrivial, so $D_i$ is supported on at most one factor of each block, or said otherwise $|\Delta_i \cap \Sigma_j| \leq 1$ for all $1 \leq j \leq r$, as required.
		
		In particular, the above claim implies that $1, \dots, s$ are contained in distinct parts of the partition $\{\Delta_1, \dots, \Delta_t\}$. By relabelling the coordinates (in a manner that preserves the block system) if necessary, without loss of generality, we may assume that $1 \in \Delta_1$, and by relabelling $\Delta_2, \dots, \Delta_t$ if necessary, without loss of generality, we may assume that $i \in \Delta_i$ for all $1 \leq i \leq s$.
		
		Second, let $S = \{ j \mid |\Delta_1 \cap \Sigma_j| = 1 \}$ and let $1 \leq i \leq s$. We claim that $S = \{ j \mid |\Delta_i \cap \Sigma_j| = 1 \}$. Since $U$ acts transitively on the set of factors of each block of $L$, there exists $u \in U$ such that $T_1^u = T_i$ and hence $D_1^u = D_i$ since $1 \in \Delta_1$ and $i \in \Delta_i$. However, $U$ acts trivially on the set of blocks of $G$, so $D_1$ and $D_i$ are supported on the same blocks, or said otherwise $S = \{ j \mid |\Delta_i \cap \Sigma_j| = 1 \}$, as required.
		
		Since $|\Delta_1 \cap \Sigma_j| \leq 1$ for all $1 \leq j \leq r$, we must have $|\Delta_1| = |S|$, so, in particular, $|S| \geq 2$. By construction, $1 \in \Delta_1$ and $1 \in \Sigma_1$, so $1 \in S$, and now fix $l > 1$ such that $l \in S$. Then $|\Delta_i \cap \Sigma_1| = |\Delta_i \cap \Sigma_l| = 1$ for all $1 \leq i \leq s$. In particular, every element of $U \cap L$ has the same number of nontrivial entries in block $T_{\Sigma_1}$ and $T_{\Sigma_l}$. However, by the shape of the element $h$, this contradicts $h^b \in U \cap L$. This contradiction has arisen from the two claims that we proved under the supposition that $s>r$, so we conclude that
		\begin{equation} \label{eq:bound_sr}
			s \leq r.
		\end{equation}
		
		Using the bounds in~\eqref{eq:bound_rn},~\eqref{eq:bound_mr} and~\eqref{eq:bound_sr}, we obtain
		\[
			|G:U| = |UL:U| = |L:U \cap L| \leq |L| = |T|^k \leq h(m)^k = h(m)^{rs} \leq h(n)^{n^2} \leq f(n),
		\]
		which completes the proof.
	\end{proof}
	
	It will be convenient to introduce some further notation for automorphism groups. Let $G$ be a finite group, let $A \leq \Aut(G)$ and let $N$ be an $A$-invariant subgroup of $G$. Write
	\[
	A|_N = \{ a|_N \mid a \in A \} \leq \Aut(N),
	\] 
	where $a|_N$ denotes the restriction of the automorphism $a$ to $N$, and similarly, write
	\[
	A|_{G/N} = \{ a|_{G/N} \mid a \in A \} \leq \Aut(G/N),
	\]
	where $a|_{G/N}$ denotes the automorphism of $G/N$ induced by $a$. Since every $A$-invariant series of $G$ can be refined to an $A$-chief series of $G$, we see that $\ell_{A|_N}(N) \leq \ell_A(G)$ with equality if and only if $N=G$, and $\ell_{A|_{G/N}}(G/N) \leq \ell_A(G)$ with equality if and only if $N = 1$. For ease of notation, we write $\ell_A(N)$ for $\ell_{A|_N}(N)$ and $\ell_{A}(G/N)$ for $\ell_{A|_{G/N}}(G/N)$.
	
	We are now in a position to prove the main result of this paper.

	\begin{proof}[Proof of Theorem~\ref{thm:main}]
	    Let $h \: \mathbb{N} \times \mathbb{N}$ be the function from Section~\ref{ss:p_bound}, let $f \: \mathbb{N} \times \mathbb{N}$ be a function that satisfies Proposition~\ref{prop:supplement}, and let $g\: \mathbb{N} \times \mathbb{N} \to \mathbb{N}$ be a function such that
		\begin{enumerate}[(i)]
			\item\label{eq1} $g(n,0) \geq 1$,
			\item\label{eq2} $g(n,c) \geq \max\{ g(n,c-1) \cdot g((n \cdot g(n,c-1))!,c-1), \, f(n) \}$ if $c \geq 1$,
			\item\label{eq3} $g$ is nondecreasing in both variables.
		\end{enumerate}
		
		Assume that $G = \bigcup_{a \in A} U^a$. For ease of notation, write $c = \ell_A(G/U_A)$. We proceed by induction on $|G|$. In the base case where $|G| = 1$, we have $|G:U| = 1 \leq g(1,0) = g(n,c)$. For the inductive step, assume that $|G| > 1$. First assume that $U_A \neq 1$. Since $G/U_A = (U/U_A)^A$ and $|G/U_A| < |G|$, by induction,
		\[
		|G:U| = |G/U_A : U/U_A| \leq g(n,c),
		\]
		noting that $|A|_{G/U_A}:\Inn(G/U_A)| \leq |A:\Inn(G)| = n$ and $\ell_A((G/U_A)/(U/U_A)_A) = \ell_A(G/U_A) = c$. Therefore, for the remainder of the proof, we may assume that $U_A = 1$. 
		
		Let $L$ be a minimal $A$-invariant subgroup of $G$ (we allow the possibility that $G = L$). If $G=UL$, then Proposition~\ref{prop:supplement} ensures that $|G:U| \leq f(n) \leq g(n,c)$. Therefore, for the remainder of the proof, we may assume that $UL < G$. 

		Since $G/L = (UL/L)^A$ and $|G/L| < |G|$, by induction, 
		\begin{equation} \label{eq:g_ul}
			|G:UL| = |G/L:UL/L| \leq g(n,c-1),
		\end{equation}
		noting that $|A_{G/L}:\Inn(G/L)| \leq |A:\Inn(G)| \leq n$ and $\ell_A(G/L) = c-1$.
		
		Consider the $A$-core $K = (UL)_A$. Certainly $K \leq UL$, so $UK \leq UL$, and as $L$ is an $A$-invariant subgroup of $UL$, we have $L \leq K$ and hence $UL \leq UK$. Therefore, $UK = UL < G$, so $K < G$ and hence $\ell_A(K) < \ell_A(G) = c$, or said otherwise,
		\begin{equation} \label{eq:bound_c}
			\ell_A(K) \leq c-1.
		\end{equation}

		For $H \leq G$, write $\overline{H} = HZ(G)/Z(G)$ considered as a subgroup of $\Inn(G) = G/Z(G)$. Then, using~\eqref{eq:g_ul},
		\[
		|A:\overline{UL}| = |A:\Inn(G)| \cdot |\Inn(G):\overline{UL}| \leq |A:\Inn(G)| \cdot |G:UL| \leq n \cdot g(n,c-1).
		\]
		Therefore, we have
		\begin{equation} \label{eq:bound_n}
			|{A|_K}:\Inn(K)| \leq |A:\overline{K}| = |A:\overline{(UL)_A}| = |A:\overline{UL}_A| \leq |A:\overline{UL}|! \leq (n \cdot g(n,c-1))!.
		\end{equation}
		
		Since $K = (U \cap K)^A$ and $|K| < |G|$, by induction, 
		\begin{equation} \label{eq:k_uk}
		|K:U \cap K| \leq g(|{A|_K}:\Inn(K)|, \ell_A(K)) \leq g((n \cdot g(n,c-1))!, c-1),
		\end{equation}
		where the final inequality follows from~\eqref{eq:bound_c} and~\eqref{eq:bound_n}.
		
		Using the bounds in~\eqref{eq:g_ul} and~\eqref{eq:k_uk}, we obtain
		\begin{align*}
			|G:U| &= |G:UK| \cdot |UK:U| = |G:UL| \cdot |K:U \cap K| \\ &\leq g(n,c-1) \cdot g((n \cdot g(n,c-1))!, c-1) \leq g(n,c). \qedhere
		\end{align*}
	\end{proof}

	\subsection*{Acknowledgements} The first and third authors are funded by the European Union via the Next Generation EU (Mission 4 Component 1 CUP B53D23009410006, PRIN 2022, 2022PSTWLB, \emph{Group Theory and Applications}). The second author is an EPSRC Postdoctoral Fellow (EP/X011879/1), and he wishes to thank the University of Milano-Bicocca for their hospitality. In order to meet institutional and research funder open access requirements, any accepted manuscript arising shall be open access under a Creative Commons Attribution (CC BY) reuse licence with zero embargo.

\end{document}